\documentclass[a4paper]{amsart}
\usepackage[utf8]{inputenc}

\usepackage{amsmath, amssymb, amsthm}
\usepackage{mathtools}  
\usepackage{bm}         
\usepackage{dsfont}	

\usepackage{placeins}   
\usepackage{xcolor}
\usepackage{microtype}  

\usepackage{longtable}  
\usepackage{booktabs}   

\usepackage[breaklinks]{hyperref}
\hypersetup{
    colorlinks,
    linkcolor={red!40!black},
    citecolor={blue!50!black},
    urlcolor={blue!80!black}
}

\newtheorem{theorem}{Theorem}

\newtheorem{lemma}[theorem]{Lemma}

\newtheorem{example}[theorem]{Example}

\theoremstyle{remark}
\newtheorem*{remark}{Remark}

\theoremstyle{definition}

\DeclareMathOperator{\tr}{tr}

\DeclareMathOperator{\sign}{sign}

\DeclareMathOperator{\imm}{imm}

\DeclareMathOperator{\per}{per}

\newcommand{\ot}[0]{\otimes}
\newcommand{\nn}[0]{\nonumber}

\newcommand{\R}{\mathds{R}}
\newcommand{\C}{\mathds{C}}

\newcommand{\one}[0]{\mathds{1}}

\newcommand{\bra}[1]{\mathinner{\langle #1|}}	
\newcommand{\ket}[1]{\mathinner{|#1\rangle}}
\newcommand{\braket}[2]{\mathinner{\langle #1|#2\rangle}}

\newcommand{\dyad}[1]{| #1\rangle \langle #1|}  	

\renewcommand{\a}{\alpha}
\renewcommand{\b}{\beta}

\newcommand{\overbar}[1]{\mkern 1.5mu\overline{\mkern-1.5mu#1\mkern-1.5mu}\mkern 1.5mu} 

\begin{document}
\title  {Matrix forms of immanant inequalities}
\date   {\today}

\author {Felix Huber}
\address{
Felix Huber,
Atomic Optics Department, 
Jagiellonian University, 
30-348 Krak\'{o}w, 
Poland, 
e-mail: felix.huber@uj.edu.pl}

\author {Hans Maassen}
\address{
Hans Maassen,
Department of Mathematics,
Radboud University, 
6525 AJ Nijmegen,
The Netherlands}

\thanks{
FH acknowledges support by
the Government of Spain (FIS2020-TRANQI and Severo Ochoa CEX2019-000910-S), 
Fundació Mir-Puig, 
Generalitat de Catalunya (AGAUR SGR 1381 and CERCA),
the European Union under Horizon2020 (PROBIST 754510),
and the Foundation for Polish Science through TEAM-NET (POIR.04.04.00-00-17C1/18-00).}

\maketitle

\noindent {\bf Abstract:} 
We translate inequalities and conjectures for immanants and generalized matrix functions into inequalities in the Löwner order.
These have the form of trace polynomials and generalize the inequalities from [FH, J. Math. Phys. 62 (2021), 2, 022203].

\section{Immanants}

Let $A = (a_{ij})$ be a complex $n\times n$ matrix. The determinant and permanent of $A$ are given by 
\begin{align}
 \det (A) &= \sum_{\sigma \in S_n} \sign(\sigma) \prod_{t=1}^n a_{t \sigma(t)} \nonumber\\
 \per (A) &= \sum_{\sigma \in S_n} \prod_{t=1}^n a_{t \sigma(t)}\,. \nn
\end{align}

A generalization of these functions are {\em immanants}. 
Given an irreducible character $\chi$ of the symmetric group $S_n$, the corresponding immanant is defined as
\begin{equation}\label{eq:immanant}
 \imm_\chi (A) = \sum_{\sigma \in S_n} \chi (\sigma) \prod_{t=1}^n a_{t \sigma(t)}\,.
\end{equation}
Now one sees that the determinant and the permanent arise 
when $\chi(\sigma) = \sign(\sigma)$  and $\chi(\sigma) = 1$ are the alternating and trivial representations respectively.

This can be generalized further: given a subgroup $H$ of $S_n$ and $\chi$ a character on~$H$,
one can define a {\em generalized matrix function}~\cite{MarcusMinc1965, Merris_MultilinearAlgebra1997}
\begin{equation}\label{eq:genmatfun}
 d^H_\chi(A) = \sum_{\sigma \in H} \chi(\sigma) \prod_{t=1}^n a_{t \sigma(t)}\,,
\end{equation}
where the sum is now over the subgroup $H$ instead of $S_n$. Note that 
$\det$, $\per$, and $\imm_\chi$ are special cases of the function $d^H_\chi$.

\bigskip
It is well known that, when $A$ is positive semidefinite, 
all the above quantities are non-negative
\begin{equation}\label{eq:perdetPSD}
d^H_\chi(A) \geq 0\,.
\end{equation}

Many more inequalities are known. A short selection appears in Appendix~\ref{app:ineqs}, 
a larger overview can be found in Chapter $7$ of Ref.~\cite{Merris_MultilinearAlgebra1997}.

\bigskip
We write $A\geq B$ whenever $A-B$ is positive semidefinite; this is the Löwner order.
The aim of this note is to turn {\em scalar} inequalities for positive semidefinite operators, 
such as e.g. appearing in Formula~\eqref{eq:perdetPSD}, 
into {\em matrix} inequalities in the Löwner order.

\section{Mapping to tensor products}
To approach this subject, it is convenient to consider any given positive semidefinite matrix~$A$ 
as the Gram matrix $A = (a_{ij})$ of some set of vectors $\{\ket{v_1}, \dots, \ket{v_n}\}$. 
So we write
\begin{equation}\label{eq:Gram}
A=
\begin{pmatrix}
  \braket{v_1}{v_1} & \braket{v_1}{v_2} &\hdots &\braket{v_1}{v_n} \\
  \braket{v_2}{v_1} & \braket{v_2}{v_2} &       & \vdots \\
  \vdots            &                   & \ddots \\
  \braket{v_n}{v_1} & \hdots                   &       & \braket{v_n}{v_n}
 \end{pmatrix}
\end{equation}
where $\ket{v_1},\dots, \ket{v_n}$ are complex vectors 
in a sufficiently large Hilbert space~$\C^m$. 
Every Gram matrix is positive semidefinite,
and in turn, every positive semidefinite
matrix A can be written as the Gram matrix of some set of vectors, which is
unique up to unitary equivalence. 
From now we call them {\em Gram vectors} for $A$. The
canonical choice will be through the matrix square root: 
for every matrix $A\geq 0$ there is a unique matrix $B\geq 0$
such that $A = B^2$. Then the columns of $B$ can be used as Gram vectors for $A$.

\smallskip
Now when $A\geq 0$ is thus understood as a Gram matrix 
one can relate immanants and generalised matrix functions of $A$ with an expression 
on a tensor product space involving its Gram vectors.
 
For this we consider the action of the symmetric group on a tensor product space. 
Namely, let $\sigma \in S_n$ act on 
$(\C^m)^{\otimes n} = \C^m \otimes \dots \otimes \C^m$ ($n$ times) 
by permuting its tensor factors,
\begin{equation}
 \sigma \ket{v_1} \ot \dots \ot \ket{v_n} = 
 \ket{v_{\sigma^{-1}(1)}} 
 \ot \dots \ot 
 \ket{v_{\sigma^{-1}(n)}}\,.\nn
\end{equation}
For example,
$
 (143)(2) \ket{v_1} \ot \ket{v_2} \ot \ket{v_3} \ot \ket{v_4}
       = \ket{v_3} \ot \ket{v_2} \ot \ket{v_4} \ot \ket{v_1}\,.
$

\smallskip

Our aim is to express the generalized matrix function $d^H_\chi$ as a multilinear function on a tensor product space;
see also Theorem $7.26$ in Ref.~\cite{Merris_MultilinearAlgebra1997} and Ref.~\cite{MaassenKuemmerer2019} 
which take such an approach.
\begin{lemma}\label{prop:gmf_gram}
 Let $A\geq 0$ with Gram vectors $\ket{v_1}, \dots, \ket{v_n}$
 and define $X_i = \dyad{v_i}$, $1\leq i \leq n$. Then
 
\begin{equation}\label{eq:d_as_projection}
 d^H_\chi(A) = \sum_{\sigma \in H} \chi(\sigma) \tr\big [\sigma^{-1} \, X_1 \ot \dots \ot X_n\big]\nn
\end{equation} 
\end{lemma}

\begin{proof}
Put $v := v_1 \ot \dots \ot v_n$. Then 
\begin{align*}
 \prod_{t=1}^n a_{t \sigma(t)} 
    &= \prod_{t=1}^n \braket{v_t}{v_{\sigma(t)}} 
     = \braket{v}{\sigma^{-1} v} \\
    &= \tr(\sigma^{-1} \dyad{v}) = \tr( \sigma^{-1} X_1 \ot \dots \ot X_n)\,.
\end{align*}
The statement follows from the definition~\eqref{eq:genmatfun} of $d^H_\chi$.
\end{proof}

Now it is interesting to understand the immanant of positive semidefinite matrices
as arising from projectors:
let $H$ be subgroup of $S_n$.
The centrally primitive idempotent in the group ring $\C[H]$ associated with a character $\chi$ is~\cite[Theorem 3.6.2]{webb_2016}
\begin{equation}\label{eq:proj_irrep}
 p^H_\chi = \frac{\chi(e)}{|H|} \sum_{\sigma \in H} \chi(\sigma) \sigma^{-1} \nn\,.
\end{equation} 
One has that $(p^H_\chi)^2= p^H_\chi$. Under a representation,
$p^H_\chi$ is projects onto an isotypic component 
and thus $p^H_\chi \geq 0$.
In terms of the above proof it is then immediate that Eq.~\eqref{eq:perdetPSD} holds:
\begin{equation}\label{eq:imm_ineq}
 d_\chi^H(A) = \frac{|H|}{\chi(e)}\braket{v}{p_\chi^H v} = \frac{|H|}{\chi(e)} \arrowvert\arrowvert p_\chi^H v\arrowvert\arrowvert^2\geq 0\,.
\end{equation}

\section{Trace polynomial inequalities}
Each summand in Lemma~\ref{prop:gmf_gram} is linear in the $X_i$'s.
This allows to turn scalar inequalities for immanants into inequalities in terms of the Löwner order. 

For this we require a canonical ordering when writing permutations.
Given a permutation decomposed into cycles as $\sigma = \sigma_1 \dots \sigma_l \in S_n$, 
its canonical order is such that its largest element appears at the end of the last cycle, $\sigma_l = (\dots n)$.
Now given a set of complex matrices $X_1, \dots, X_n$ of equal size, 
we define for every permutation $(\a_1 \dots \alpha_r) \dots (\zeta_1 \dots \zeta_t)(\xi_1  \dots \xi_{w} n)$ 
the scalar
\begin{align*}
 T_\sigma &= \tr(X_{\a_1} \dots X_{\alpha_r}) \cdots \tr(X_{\zeta_1} \dots X_{\zeta_t}) 
 \tr(X_{\xi_1} \cdots X_{\xi_\upsilon} X_{n}) \,.
\intertext{and the matrix}
 \widetilde T_\sigma  &= \tr(X_{\a_1} \dots X_{\alpha_r}) \cdots \tr(X_{\zeta_1} \dots X_{\zeta_t}) X_{\xi_1} \cdots X_{\xi_\upsilon}\one \,,
\end{align*}
where the last matrix $X_n$ is replaced by $\one$ and its trace opened.
It is clear that $T_\sigma = \tr(\widetilde T_\sigma X_n)$ holds.
Furthermore, it is well-known that~\cite{Kostant2009}
\begin{equation*}\label{eq:perm_to_mult}
 T_\sigma = \tr(\sigma^{-1} X_1 \ot \dots \ot X_n)
\end{equation*}

\bigskip
Now suppose one is given a function $f:S_n \to \C$. To it we associate, on the one hand, the generalized matrix function
$d_f:M_n\to \C$,
\begin{equation*}
 d_f(A) = \sum_{\sigma \in S_n} f(\sigma) \prod_{t=1}^n a_{t\sigma(t)}\,.
\end{equation*}

On the other hand, we define for complex $m \times m$ matrices $X_1, \dots, X_{n-1}$ the matrix-valued map
$\widetilde{d}_f:M_m^{n-1} \to M_m$,
\begin{equation*}
\widetilde d_f(X_1, \dots, X_{n-1}) = \sum_{\sigma \in S_n} f(\sigma) \widetilde{T}_\sigma\,.
\end{equation*}

Starting from an inequality for some $d_f$ a corresponding matrix inequality for $\widetilde d_f$ can be made.

\begin{theorem}\label{thm:gmf}
 Suppose for all positive semidefinite $n \times n$ matrices $A$ it holds that
 \begin{equation*}
  d_f(A) \geq 0\,.
 \end{equation*}
Then for all positive semidefinite matrices $X_1, \dots, X_{n-1}$ of equal size it holds that
\begin{equation*}\label{eq:pCH}
  \widetilde{d}_f(X_1, \dots, X_{n-1}) \geq 0 \,.
\end{equation*}
\end{theorem}

\begin{proof}
Suppose the $X_i$'s have size $m \times m$.
Then 
$
  \widetilde{d}_f \geq 0 
$
if and only if 
$
 \bra{w} \widetilde{d}_f \ket{w} \geq 0
$
holds for all $\ket{w} \in \C^m$. We rewrite this as
\begin{align}\label{eq:gmf_rewrite}
 \bra{w} \widetilde{d}_f \ket{w} 
  &= \bra{w} \sum_{\sigma \in S_n} f(\sigma) \widetilde{T}_\sigma \ket{w} \nn\\
  &= \sum_{\sigma \in S_n} f(\sigma) T_\sigma \nn\\
  &= \sum_{\sigma \in S_n} f(\sigma) \tr(\sigma^{-1} X_1 \ot \dots \ot X_n)\,,
\end{align} 
where we have set $X_n = \dyad{w}$.

Now every such tensor product $X_1 \ot \dots \ot X_{n}$ where all $X_i$ are positive semidefinite can be decomposed 
as a convex combination of a finite set of rank one matrices of the form 
$\dyad{v_1} \ot \dots \ot \dyad{v_n}$. 
It thus suffices to show that Eq.~\eqref{eq:gmf_rewrite} holds for each term in the decomposition.
With Lemma~\ref{prop:gmf_gram} and Eq.~\eqref{eq:proj_irrep} we see that 
\begin{equation}
 \sum_{\sigma \in S_n} f(\sigma) \tr(\sigma^{-1} \dyad{v_1} \ot \dots \ot \dyad{v_n}) 
     =  d_f(A) \geq 0\nn\,,
\end{equation} 
where $A$ is the Gram matrix of the vectors $\ket{v_1}, \dots, \ket{v_n}$.

This ends the proof.
\end{proof}

In particular, if for some $\a, \b \in \R$ an inequality
$\a d_\chi^H(A) \geq \beta d^K_\psi(A)$ holds for all complex positive semidefinite $n \times n $ matrices $A$,
then 
$\a \widetilde{d}_\chi^H(X_1, \dots, X_{n-1}) \geq \beta \widetilde{d}^K_\psi(X_1, \dots, X_{n-1})$ for all $X_1, \dots, X_{n-1} \geq 0$ of 
equal size.

\begin{remark}
In Ref.~\cite{huber2020positive} the author has derived a special case of this inequality, namely when $\chi$ is a character of the symmetric group, 
in the framework of the Choi-Jamio\l kowski isomorphism.
\end{remark}

\begin{example}
Let $H = A_4$ be the alternating group of degree four.
The conjugacy class of $(123)$ in $S_4$ splits 
in $A_4$ into two: the conjugacy classes
\begin{align*}
 C_{A_4}((123)) &= \{(123), (142), (243), (134)\} \,,\\
 C_{A_4}((132)) &= \{(132), (124), (234), (143)\} \,.
\end{align*}
The conjugacy class of $(12)(34)$ does not split.
So one has the representatives
$(e)\,, (12)(34)\,, (123)\,, (132)$. 
Define $\omega = e^{2\pi i/3}$. The non-trivial characters of $A_4$ are 

\begin{table}[h]
\begin{tabular}{@{}ccccc@{}}
 \toprule 
          & $(e)$ & $(12)(34)$ & $(123)$      & $(132)$ \\
 \midrule
 $\chi_1$ & $3$  & $-1$   & $0$        & $0$   \\
 $\chi_2$ & $1$   & $1$   & $\omega$   & $\overbar{\omega}$ \\
 $\chi_3$ & $1$   & $1$   & $\overbar{\omega}$ & $\omega$ \\
 \bottomrule
 \end{tabular}
\end{table}
The element $\frac{\chi_i(e)}{|A_4|} \sum_{\sigma \in A_4} \chi_i(\sigma^{-1}) \sigma$ is idempotent and thus $d_{\chi_i}^{A_4}(A) \geq 0$ for all $A\geq 0$ as reasoned in Eq.~\eqref{eq:imm_ineq}.
The expressions are multilinear and it sufficices to consider matrices of trace one.
Then for square matrices $X,Y,Z$ define
\begin{align*}
 L &= \tr(XY)Z + \tr(XZ)Y + \tr(YZ)X\,,  \\
 M &= \tr(ZYX)\one + XY + YZ + ZX\,.
\end{align*}
The matrix functions $\widetilde{d}_{\chi_i}^{A_4} = \widetilde{d}_{\chi_i}^{A_4}(X,Y,Z)$ are then
\begin{align*}
\widetilde{d}_{\chi_1}^{A_4} &= 3\one - L\\
\widetilde{d}_{\chi_2}^{A_4} &=  \one + L + \omega M + \overbar{\omega} M^*  \\
\widetilde{d}_{\chi_3}^{A_4} &=  \one + L + \overbar{\omega} M + \omega M^* \,.
\end{align*}
where $M^*$ is the hermitian conjugate of $M$.
Theorem~\ref{thm:gmf} implies that for all $X,Y,Z\geq 0$ with $\tr(X)=\tr(Y)=\tr(Z)=1$ 
the above three expressions are positive semidefinite.
\end{example}

\begin{example}
 Set $f = \chi_2^{A_4}$. Again it suffices to consider matrices of trace one.
 It follows from Theorem~\ref{thm:gmf} and Watkins's theorem (see Appendix~\ref{app:ineqs}) that 
 \begin{equation*}
  (\omega - 1) M + (\overbar{\omega} - 1) M^* +  \tr(L)\one + X + Y + Z 
   + X\{Y,Z\} + Y\{Z,X\} + Z\{X,Y\} \geq 0
 \end{equation*} 
 holds for all $X,Y,Z \geq 0$ with $\tr(X)=\tr(Y)=\tr(Z)=1$.
 \end{example}

 \begin{example}
  In $\C S_3$ the centrally primitive idempotent associated to the partition $(2,1)$ reads 
  $p_{(2,1)} = \frac{1}{3}[2(e) - (123) - (132)]$.  
  Watkin's theorem applied to $p_{(1,2)}$ and the non-negativity of the determinant yields the following upper and lower 
  bounds on the anti-commutator of positive semidefinite matrices $X,Y$ of trace $1$:
  \begin{equation*}
   X + Y + [\tr(XY) -1] \one \leq XY + YX \leq \frac{2}{3}[X + Y + \tr(XY)]\one\,.
  \end{equation*}
 \end{example}

 \begin{remark}
 We note when $\widetilde{\imm}_\chi := \widetilde{d}_\chi^{S_n}(X_1, \dots, X_{n-1})$ becomes a matrix {\em identity}: 
 one has that $\widetilde{\imm}_\chi(X_1, \dots, X_{n-1}) = 0 $ for all complex $m\times m$ matrices $X_1, \dots, X_{n-1}$ 
 when $\chi$ corresponds to a Young diagram with more than $m$ rows. 
 This result is also known as Lew's polarized (or generalized) Cayley-Hamilton theorem~\cite{Lew1966} 
 and it can be understood as a consequence of  
 \begin{equation*}
   p_\chi \ket{v_1} \ot \dots \ot \ket{v_n} = 0\,.
 \end{equation*}
 whenever $\chi$ corresponds to a Young diagram with $r$ rows but 
 $\ket{v_1}, \dots, \ket{v_n} \in \C^m$ with $m<r$~\cite{huber2020positive}.
 \end{remark}

\appendix
\section{Some immanant inequalities}
\label{app:ineqs}

The following inequalities for the immanants of positive semidefinite matrices are 
well-known~\cite{CHEON2005314, grone1988, Merris_MultilinearAlgebra1997, Zhang2016}.
\smallskip

\noindent Hadamard inequality (1893): For all $A\geq 0$
 \begin{equation*} 
  \det(A) \leq \prod_i a_{ii} \,.
 \end{equation*}
 
\noindent  Schur inequality (1918):  For all $A\geq 0$
 \begin{equation*}
  \det(A) \leq \frac{\imm_\lambda(A)}{\chi_\lambda(e)} \,.
 \end{equation*}
 
\noindent Marcus inequality (1963): For all $A\geq 0$
 \begin{equation*}
  \per(A) \geq \prod_i a_{ii}\,.
 \end{equation*}
 
\noindent   Permanent dominance conjecture (open):
 For all $A\geq 0$
 \begin{equation*}
  \per(A) \geq \frac{\imm_\lambda(A)}{\chi_\lambda(e)} \,.
 \end{equation*}  
 
\noindent Heyfron's theorem (1988):
 Define the normalized immanant $\overbar{\imm}_\chi(A) = \imm_\chi(A) / \chi(e)$.
 For all $A\geq 0$
 \begin{align*}
  \det(A) = \overbar{\imm}_{[1^n]}(A) &\leq \overbar{\imm}_{[2, 1^{n-2}]}(A) \leq \overbar{\imm}_{[3, 1^{n-3}]}(A) \leq  \dots \\  
  &\dots \leq \overbar{\imm}_{[n-3, 1^3]}(A) \leq \overbar{\imm}_{[n-2, 1^2]}(A) \leq \overbar{\imm}_{[n]}(A) = \per(A)
 \end{align*} 
 where the dots consist of a chain of immanants corresponding to single-hook Young Tableaux 
 with increasing row and decreasing column lenghts.
 
\bigskip
\noindent  Watkins's theorem (1988):
 Suppose $f: S_n \to \C$ is such that $d_f(A) \geq 0$  for all complex semidefinite $m \times m$ matrices $A$. 
 Then it holds that $d_f(A) \geq f(e) \det(A)$
 for all complex semidefinite $m \times m$ matrices $A$.

\bibliographystyle{amsplain}
\bibliography{current_bib}
\end{document}